\documentclass[12pt,a4paper]{amsart}
\usepackage{amsfonts}
\usepackage{amsthm}
\usepackage{amsmath}
\usepackage{amscd}
\usepackage[latin2]{inputenc}
\usepackage{t1enc}
\usepackage[mathscr]{eucal}
\usepackage{indentfirst}
\usepackage{graphicx}
\usepackage{graphics}
\usepackage{pict2e}
\usepackage{epic}
\usepackage{verbatim}
\usepackage[colorlinks,linkcolor=red,anchorcolor=blue,citecolor=green,CJKbookmarks=True]{hyperref}
\usepackage[margin=3.8cm]{geometry}
\newtheorem{Th}{Theorem}[section]
\newtheorem{Lemma}[Th]{Lemma}
\newtheorem{Cor}[Th]{Corollary}

 \theoremstyle{definition}
\newtheorem{Def}[Th]{Definition}

\newtheorem*{rem}{Remark}

\newtheorem{?}[Th]{Problem}

\begin{document}
\author{Jingze Zhu}
\title{quantitative long range curvature estimate for mean curvature flow}
\bibliographystyle{amsalpha}
\begin{abstract}
  We prove that smooth convex $\alpha$-noncollapsed ancient mean curvature flow satisfies a quantitative curvature estimate $H(y,t)\leq CH(x,t)(H(x,t)|x-y|+1)^2$ for any pair of $x,y$.
  In other words, the rescaled curvature grows at most quadratically in terms of the rescaled extrinsic distance.
\end{abstract}
\address{Department of Mathematics, Columbia University, New York, NY, 10027}
\email{zhujz@math.columbia.edu}
\maketitle
\section{Introduction}
In the study of mean convex mean curvature flow, the smooth convex $\alpha$-noncollapsed ancient solutions play a crucial role, as they model the singularities of the mean convex flow. Such results were first proved by the seminal work
of White \cite{white2000size, white2003nature}. Sheng-Wang \cite{sheng2009singularity} later gave an alternative proof.

In 2013, Haslhofer-Kleiner \cite{haslhofer2017} found an interesting way to significantly simply the theory. They made the noncollapsedness an assumption, not the consequence. This is a crucial difference from the earlier work.
This assumption is reasonable because of the direct proof of noncollapsing property by Andrews \cite{andrews2012noncollapsing} using maximum principle.
One of the important results in \cite{haslhofer2017} is the long range curvature estimate.
For our purpose, we state a simplified version of this estimate:
\begin{Th}[c.f \cite{haslhofer2017} Theorem 1.10, Corollary 3.2]
\label{H-K Thm}
  Given $\alpha\in (0,1]$, there is an increasing function $\varphi:\mathbb{R}_{\geq 0}\rightarrow\mathbb{R}_{+}$ such that
  for any smooth convex ancient $\alpha$-noncollapsed mean curvature flow $M^n_t\subset\mathbb{R}^{n+1}$  and $x,y\in M_t$ we have
  \begin{align*}
     H(y,t) \leq H(x,t)\varphi(H(x,t)|x-y|)
  \end{align*}
\end{Th}

Roughly speaking, Haslhofer-Kleiner's theorem allows one to compare the curvature uniformly in terms of their rescaled extrinsic distance. However, the control function $\varphi$ might be growing very quickly near infinity.

The goal of this note is to show that $\varphi$ grows at most quadratically, i.e, we can take $\varphi(s)=C(s+1)^2$. Here is our main theorem:
 \begin{Th}\label{Main Thm}
  Given $\alpha\in (0,1]$, there is a constant $C=C(\alpha, n)$ such that
  for any smooth convex ancient $\alpha$-noncollapsed mean curvature flow $M^n_t\subset\mathbb{R}^{n+1}$ and $x,y\in M_t$ we have
  \begin{align}\label{Long range curvature estimate}
     H(y,t) \leq CH(x,t)\left(1+H(x,t)|x-y|\right)^2
  \end{align}
\end{Th}

By switching $x,y$, We can also get a quantitative curvature decay rate estimate:
\begin{Cor}\label{Main Cor}
  Setting up as in the Theorem \ref{Main Thm}, we can find $c=c(\alpha,n)>0$ such that
  \begin{align*}
    H(y,t) \geq c H(x,t)\left(1+H(x,t)|x-y|\right)^{-\frac{2}{3}}
  \end{align*}
\end{Cor}

The proof of Theorem \ref{Main Thm} relies on the Ecker-Huisken's
interior estimate \cite{ecker1991interior}. The convexity and noncollapsing property gives that
the surface is graphical in a small ambient ball $B_r(y)$.
We need to figure out the graphical radius $r$ carefully (it could be very small when $|x-y|$ is large) and then use Ecker-Huisken's interior estimate.

We want to mention that part of the argument is similar to those in \cite{brendle2017fully} (Section 5).
\newline

\noindent\textbf{Acknowledgements } The author would like to thank his advisor Simon Brendle for giving insightful ideas.

\section{Proof of the results}
Let $M^n_t$ be a smooth
ancient solution of mean curvature flow which bounds the convex domain $K_t$ in $\mathbb{R}^{n+1}$ and is $\alpha$-noncollapsed.

We recall the definition of the noncollapsing property:
\begin{Def}[c.f \cite{sheng2009singularity}, \cite{andrews2012noncollapsing}]
Let $M$ be a smooth mean convex hypersurface bounding a domain $K$ in $\mathbb{R}^{n+1}$.
Then $M$ is $\alpha$-noncollapsed, if for each $x\in M$ there is an interior ball $B\subset K$ and an exterior open ball $B'\subset K^c$ of radius $\alpha H(x)^{-1}$ with $x\in\partial B$ and $x\in\partial B'$.
\end{Def}

We need an elementary lemma for convex set:
\begin{Lemma}\label{Lemma1}
  Suppose that $K$ is a convex domain in $\mathbb{R}^{n+1}$ containing a ball $B_r(0)$  and
  $x\in \partial K$. Let $\omega = \frac{x}{|x|}$, then for any supporting plane $P$ at $x$ we have $\langle\nu,\omega\rangle\geq \frac{r}{|x|}$, where $\nu$ is the outward unit normal vector of $P$.
\end{Lemma}
\begin{proof}
  By the definition of the supporting hyperplane, $K$ lies in one side of $P$, consequently we have $\langle x-y,\nu\rangle\geq 0$ for any $y\in K$.
  Since $B_r(0)\subset K$, we can take $y = r\nu$, then the result follows.
\end{proof}

\begin{rem}
  If $\partial K$ is smooth, then $\nu$ is the
  outward normal of the tangent plane of $\partial K$ at $x$.
\end{rem}

\begin{proof}[Proof of Theorem \ref{Main Thm}]
  We may assume that the mean curvature is positive, for otherwise Theorem \ref{Main Thm} is equivalent to Theorem \ref{H-K Thm} (or by strong maximum principle the solution is flat)

  Since (\ref{Long range curvature estimate}) is scale invariant, we may assume without loss of generality that $H(x,t)=1$.

  By $\alpha$-noncollapsing assumption, there is a ball $B_{\alpha}(p)\subset K_t$ that is tangential to $M_t$ at $x$.

  Let $L=|y-p|\geq \alpha$ and $\omega = \frac{y-p}{|y-p|}$.

  Since the flow is mean convex, $K_t$ is a nonincreasing set.
  Consequently, $B_{\alpha}(p)\subset K_s$ for any $s\leq t$.
  Hence, for any $z\in M_s\cap B_{\frac{\alpha}{6}}(y)$ with $s\leq t$, we can apply Lemma \ref{Lemma1} to obtain
  \begin{align}\label{radial inner product}
    \langle \frac{z-p}{|z-p|}, \nu(z)\rangle\geq \frac{\alpha}{|z-p|}\geq \frac{5\alpha}{6L}
  \end{align}
  Moreover, we have the following elementary inequality:
  \begin{align}\label{radial inner product difference}
    \left|\frac{z-p}{|z-p|}-\omega\right| = & \left|\frac{z-p}{|z-p|}-\frac{y-p}{|y-p|}\right|\\
    \leq& \frac{|z-y|}{|y-p|} + \left|1-\frac{|z-p|}{|y-p|}\right| \notag\\
    \leq& \frac{\alpha}{3L} \notag
  \end{align}

  Combining (\ref{radial inner product}) and (\ref{radial inner product difference}) we have
  \begin{align}\label{nu omega}
    \langle \omega, \nu(z)\rangle \geq \frac{\alpha}{2L}
  \end{align}

  In particular, $B_{\frac{\alpha}{6}}(y)\cap M_s$ is a graph over
  the plane perpendicular to $\omega$, whenever $s\leq t$.

  Now we let
  $R=\frac{\alpha}{6}$, $t_0 = t - \frac{\alpha^2}{400n}$ and

  $$K(y,s,R) = \{z\in M_{s} \big| |z-y|^2+2n(s-t_0)\leq R^2\}$$.

  By Ecker-Huisken's interior estimate (c.f \cite{ecker1991interior} Corollary 3.2):
  \begin{align}\label{Ecker Huisken Interior}
    \sup\limits_{K(y,t,R/2)} |A|^2\leq C(n) (R^{-2}+(t-t_0)^{-1}) \sup\limits_{s\in[t_0,t]}\sup\limits_{K(y,s,R)} \langle \omega, \nu\rangle^{-4}
  \end{align}
  Note that in the left hand side, $(y,t)\in  K(y,t, R/2)$.
  In the right hand side, each time slice is contained in $B_{\frac{\alpha}{6}}(y)$.
  Hence by (\ref{nu omega}) (\ref{Ecker Huisken Interior}) we have:
  \begin{align}
     |A(y,t)|\leq C(n)\alpha^{-1} \left(\frac{2L}{\alpha}\right)^2 = C(n)\alpha^{-3}L^2
  \end{align}

  Note that by convexity $0\leq H\leq \sqrt{n}|A|$. Moreover, $L=|y-p|\leq |x-y|+\alpha$.
  Consequently,
  \begin{align*}
    H(y,t)\leq C(n,\alpha)H(x,t) (1+H(x,t)|x-y|)^2
  \end{align*}
  The theorem is proved.

\end{proof}

We now describe how to derive Corollary \ref{Main Cor}:
First by (\ref{Long range curvature estimate}) we see that
either $H(y,t)\leq CH(x,t)$ or $H(x,t)|x-y|\geq 1$, hence:
\begin{align}\label{Formula1}
  H(y,t) \leq CH(x,t) + H(x,t)H(y,t)|x-y|
\end{align}
It follows that
\begin{align}\label{Formula2}
  \frac{1+H(x,t)|x-y|}{1+H(y,t)|x-y|}H(y,t)\leq CH(x,t)
\end{align}
Here $C$ might be different from line to line, but only depends on $n,\alpha$.

Next, combining (\ref{Long range curvature estimate}) and (\ref{Formula2}) we obtain:
\begin{align}\label{Formula3}
   &H(y,t)\left(\frac{H(y,t)}{1+H(y,t)|x-y|}\right)^2 \\
   \leq &CH(x,t)\left(\frac{1+H(x,t)|x-y|}{1+H(y,t)|x-y|}H(y,t)\right)^2 \notag \\
   \leq &CH(x,t)^3 \notag
\end{align}
The result follows from swapping $x,y$ in (\ref{Formula3}) and taking cubic root.
\bibliography{BIBQUANTITAVIECURVATURE}

\end{document}